\documentclass[11pt]{amsart}
\usepackage{amsmath}
\usepackage{amsthm}
\usepackage{amssymb}
\usepackage{amscd}
\usepackage{amsfonts}
\usepackage{amsbsy,subeqnarray}
\usepackage{epsfig}
\usepackage{pdfsync, hyperref,epstopdf}
\usepackage{float}
\newtheorem{remark}{Remark}
\newtheorem{proposition}{Proposition}
\newtheorem{lemma}[proposition]{Lemma}

\newtheorem{fact}{Fact}

\newtheorem*{claim*}{Claim}
\newtheorem{definition}[proposition]{Definition}

\def\eg{{\em e.g.\ }}

\newfont\bbf{msbm10 at 12pt}
\def\eps{\varepsilon}

\def\R{{\mathbb R}}

\def\N{{\mathbb N}}
\def\Z{{\mathbb Z}}

\def\le{\leqslant}
\def\ge{\geqslant}

\def\1{ {\hbox{{\it 1}} \!\! I} }

\newdimen\AAdi%
\newbox\AAbo%
\def\AArm{\fam0 }
\def\AAk#1#2{\setbox\AAbo=\hbox{#2}\AAdi=\wd\AAbo\kern#1\AAdi{}}%
\def\AAr#1#2#3{\setbox\AAbo=\hbox{#2}\AAdi=\ht\AAbo\raise#1\AAdi\hbox{#3}}%

\def\BBone{{\AArm 1\AAk{-.8}{I}I}}%

\newcommand {\CA}{{\mathcal A}}

\newcommand {\CC}{{\mathcal C}}

\newcommand {\CL}{{\mathcal L}}

\newcommand {\CP}{{\mathcal P}}

\newcommand {\CV}{{\mathcal V}}

\newcommand{\disp}{\displaystyle}
\newcommand{\8}{\infty}

\def\m1{{-1}}

\newcommand{\ninf}{{n\rightarrow\8}}

\def\S{\Sigma}
\def\s{\sigma}

\newcommand{\wh}{\widehat}
\newcommand{\wt}{\widetilde}

\def\be{\beta}
\def\al{\alpha}\def\ga{\gamma}
\def\de{\delta}

\begin{document}
\synctex=1
\floatplacement{figure}{H}
\title[non-flat  Phase Transitions]
{Chaos : butterflies also generate phase transitions and parallel universes}
\author{Renaud Leplaideur}
\date{Version of \today}
\thanks{Part of this research was supported by PUC-Santiago
}



\begin{abstract}
We exhibit examples of mixing subshifts of finite type and potentials such that there are phase transitions but the pressure is always strictly convex. More surprisingly, we show that  the pressure can be analytic on some interval although there exist several equilibrium states. 
\end{abstract}

\maketitle

\section{Introduction}\label{sec:intro}

\subsection{General background}
In this paper we deal with the notion of phase transition. There are several viewpoints for studying phase transitions: Statistical Mechanics, Probability Theory or Dynamical Systems. Depending on the viewpoint, the settings, the questions and the interests are different. In Statistical Mechanics and in Probability,  one usually considers lattices with interaction energy between the sites. Often, the geometry of the lattice and the decay of correlation of interactions are the main issues. 
On the other hand, in Dynamical Systems, one usually consider  one-dimensional lattice (with natural $\Z$-actions) and the potential as a  function of which regularity is the main issue. 

\medskip
The definition of phase transition also naturally depends on the viewpoint. 
The Ising model (see \eg \cite{georgii}) and its extension, the so-called Pots model are studied in Probability Theory, either in mean fields  case (see \cite{Ellis1, Ellis-book, Bovier}), or via the percolation theory (see \cite{grimmett, cerf}). In both cases, a phase transition means co-existence of several probability measures resolving  or resulting from some optimization.  \\
In very applied Physics, a phase transition ``simply''  is,  for instance, the boiling water. In more theoretical physics, it is related to  the G\^{a}teaux differentiability  of the pressure function (see \eg \cite{daniels-vanenter2, daniels-vanenter1,phelps,ruelle}). 
\\
The topic is actually relatively new in Dynamical Systems (compared to Statistical mechanics and Probability Theory), and one usually considers that a phase transition occurs when the pressure function stops to be analytic (see \cite{makarov-smirnov, dobbs,iommi-todd}). This definition was also used in Statistical Mechanics via the Ehrenfest classification, and the lack of analyticity defines notion of {\em first-order}, {\em second over} or even higher order phase transitions.

We want to emphasize here that, both definitions (regularity of the pressure v.s. co-existence of several equilibria) are not necessarily well-related. Of course, a first-order phase transition, that is when the pressure function is not $\CC^{1}$, yields co-existence of  several equilibria.  But we show here that the converse is not true. Actually, our Theorem B states that  there are mixing systems such that the pressure function is analytic on some interval despite simultaneous existence of  several equilibria.

\medskip
One motivation in Dynamical System for studying phase transition is ergodic optimization. Since the 00's, mathematicians somehow rediscovered the notion of {\em ground states} well known in Statistical Mechanics: if $T:X\to X$ is the dynamical system and $\varphi:X\to \R$ the potential, one studies what happens to the/some equilibrium $\mu_{\be}$ for $\be.\varphi$  as $\be$ goes to $+\8$. in Statistical Mechanics $\be$ is the inverse of the temperature. Most of the phase transitions studied in Dynamical Systems are actually ``freezing'' phase transitions, that is that for $\be>\be_{c}$ the pressure function is affine. \\
Such transitions are known in Statistical Mechanics as the  Fisher-Felderhof models (see \eg \cite{fisher}, and also \cite{dyson} for a one-dimensional lattice case). Physically, this means that for some positive temperature, the systems reaches its/one ground state and then stops to change. \\
It was thus natural to inquire about the possibility to get phase transitions in Dynamical Systems which are not freezing, that is, that after the transition the pressure is non-flat. This is the purpose of our Theorem A.


\subsection{General settings}

We consider a subshift  of finite type $\S$  on a finite alphabet $\CA$. Several cases will be considered depending on the theorem.

%
%
%

We remind that a point $x$ in $\S$ is a sequence $x_{0},x_{1},\ldots$ (also called an infinite word) where $x_{i}$ are letters in $\CA$. Moreover, there is an incidence matrix which makes precise what the authorized transition $x_{i}\to x_{i+1}$ are.

The distance between two points $x=x_{0},x_{1},\ldots$ and $y=y_{0},y_{1},\ldots$ is given by 
$$d(x,y)=\frac1{2^{\min\{n,\ x_{n}\neq y_{n}\}}}.$$
A finite string of symbols $x_{0}\ldots x_{n-1}$ is also called a \emph{word}, of length $n$. For a word $w$, its length is $|w|$. A \emph{cylinder} (of length $n$) is denoted by $[x_{0}\ldots x_{n-1}]$. It is the set of points $y$ such that $y_{i}=x_{i}$ for $i=0,\ldots n-1$. 

If $i$ is a digit in $\CA$, $x=i^{n}*$ means that $\disp x=\underbrace{i\ldots i}_{n\text{ digits}}j$ with $j$ any digit $\neq i$ such that $ij$ is admissible in $\S$.

The alphabet will depend on some integer parameter $L$. It will be either 
$$\{1,2,3,4,1_{1},1_{2},\ldots, 1_{L}\}\text{ or }\{1,2,3,4, 3',4',1_{1},1_{2},\ldots, 1_{L}\},$$
and $L$ may be equal to 1.

Let consider positive real numbers $\alpha$, $\gamma$, $\delta$ and $\eps$. The potential $\phi$ is defined by 
$$\phi(x)=\begin{cases}
-\alpha<0\text{ if }x_{0}\in\{1,1_{1},1_{2},\ldots, 1_{p}\}\\
-\log\left(\frac{n+1}n\right)\text{ if } n+1=\min\{j\ge 1,\ x_{j}\neq 2\}\\
\gamma -\eps\log\left(\frac{n+1}n\right)\text{ if } x_{0}=3, 3' \text{ and } \text{ and } n+1=\min\{j\ge 1,\ x_{j}=2\},\\
\gamma+\delta -\eps\log\left(\frac{n+1}n\right)\text{ if } x_{0}=4,4' \text{ and } \text{ and } n+1=\min\{j\ge 1,\ x_{j}=2\}.
\end{cases}$$
Hence, if $x=2^{n}*$, $\phi(x)=-\log\left(\frac{n+1}{n}\right)$. If $x=3\disp\underbrace{x_{1}\ldots x_{n-2}}_{n-2\text{ digits} =3,4}32\ldots$, then $\phi(x)=\disp \gamma-\eps\log\left(\frac{n+1}{n}\right)$. 


We recall that for $\be\ge 0$, the pressure function $\CP(\be)$ is defined by 
$$\CP(\be)=\max_{\mu\ T-inv}\left\{h_{\mu}+\be\int\phi\,d\mu\right\},$$
where $h_{\mu}$ is the Kolmogorov entropy of the measure $\mu$. We refer the reader to \cite{bowen} for classical results on thermodynamic formalism in the shift. A measure which realizes the maximum in the above equality is called an \emph{equilibrium state} for $\be.\phi$. 

\begin{definition}
\label{def-phasetransi}
We say that $\CP(\be)$ (or equivalently that the potential $\phi$) has a phase transition (at $\be_{c}$) if $\be\mapsto\CP(\be)$ is not analytic at $\be=\be_{c}$. 
\end{definition}

\subsection{Results}
We first consider the case $\disp \CA=\{1,2,3,4,1_{1},1_{2},\ldots, 1_{L}\}$. The transitions are given by Graph \ref{Fig-transibutterfly}. This gives a ``butterfly'' with two wings, each one tending to be autonomous.

\begin{figure}[htbp]
\begin{center}
\includegraphics[scale=0.5]{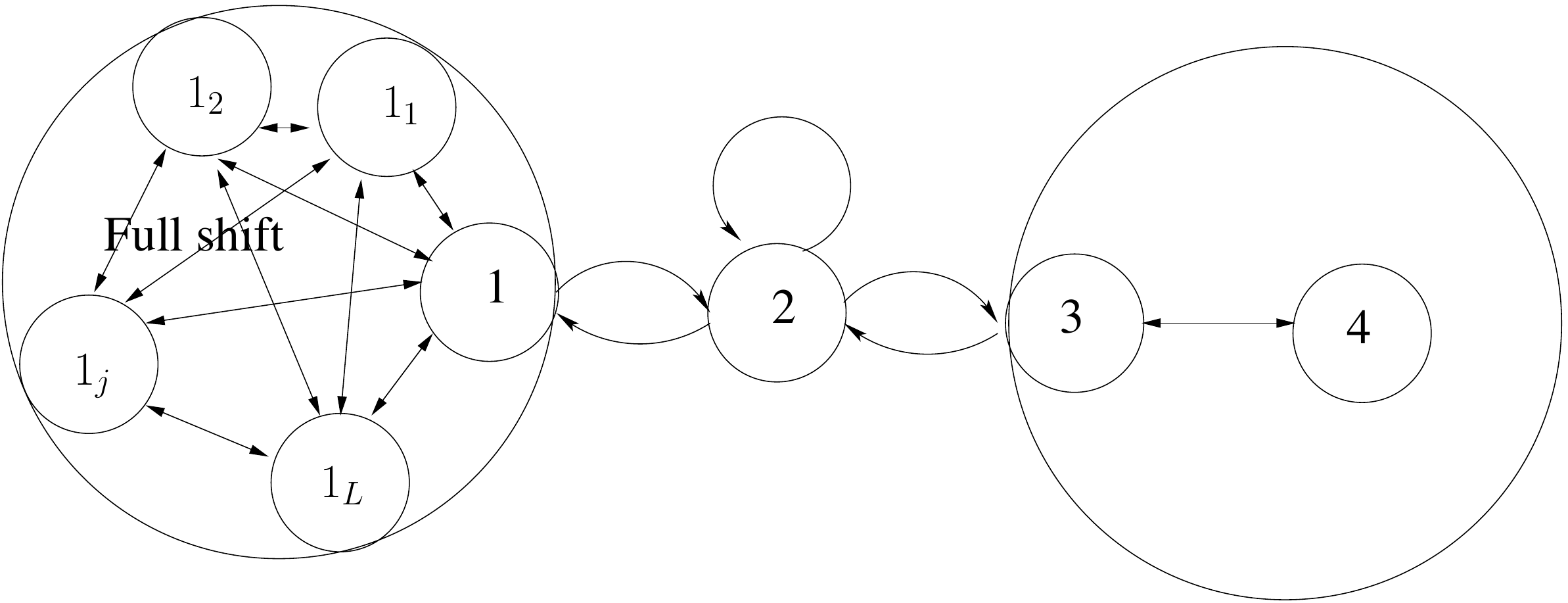}
\caption{Dynamics for Theorem A.}
\label{Fig-transibutterfly}
\end{center}
\end{figure}

We emphasize that the system is irreducible but has several subsystems. In particular we shall consider two different subsystems, $\S_{34}:=\{3,4\}^{\N}\subset\S$ and $\S_{234}:=\{2,3,4\}^{\N}\cap\S$,  the restriction of $\S$ to the invariant set of infinite words only with letters 2, 3 or 4. For the same potential $\phi$, we shall consider the associated pressure functions, $\CP_{34}(\be)$ and $\CP_{234}(\be)$.

\medskip\noindent
{\bf Theorem A.}\\{\it There  exist two positive real numbers $\be_{1}< \be_{c}$ such that $\CP(\be)$ and $\CP_{234}(\be)$ have phase transition respectively at $\be_{c}$ and $\be_{1}$. More precisely (see Figure \ref{fig-curves} p.\pageref{fig-curves}), 
\begin{enumerate}
\item the pressure function $\CP_{34}(\be)$ is analytic,
\item for $\be<\be_{c}$, the pressure function $\CP(\be)$ is analytic  and satisfies $\CP(\be)>\CP_{34}(\be)$,
\item for $\be\ge \be_{c}$, the pressure function $\CP(\be)$ satisfies $\CP(\be)=\CP_{34}(\be)$,
\item for $\be<\be_{1}$, the pressure function $\CP_{234}(\be)$ is analytic  and satisfies $\CP_{234}(\be)>\CP_{34}(\be)$,
\item for $\be\ge \be_{1}$, the pressure function $\CP_{234}(\be)$ satisfies $\CP_{234}(\be)=\CP_{34}(\be)$.
\end{enumerate}
Moreover, for each transition, the entropy is positive.

For $\be=\be_{1}$, there is an unique equilibrium state for $\S_{234}$. For $\be=\be_{c}$, there are two equilibrium states for $\S$ if and only if $\eps\be_{c}>2$. This inequality can be realized (depending of the parameters $\al,\eps,\ga$ and $\de$) or non-realized. 
}

\medskip
\noindent
As we said above, the principal motivation for Theorem A was to built phase transitions with non-flat pressure function after the transition. Actually such an example was already known in \cite{diaz-gelfert-rams}. However, in that case, the map is a skew product over a Horseshoe and the potential the logarithm of the derivative in the central direction (equivalent to the direction of fibers). Then, when it is projected onto the Horseshoe, the potential is not a function anymore. Hence, that phase transition cannot be realized as a continuous potential defined on a subshift of finite type.

To be complete concerning the shape of the pressure, we also have to mention \cite{iommi-todd}. There,  they also show that the pressure may be non-flat after the transition. However, and this is the main difference, in their construction they always need some interval where the pressure  is flat. Following their vocabulary, for this interval the systems is transient, which means it has no conformal measure. In our example, at any $\be$ there exists a conformal measure and the pressure is always strictly convex.

\medskip
Let us now present the main result of the paper. As we said above, the regularity of the pressure function is an indicator of the uniqueness of the equilibrium state. This regularity can be understood in two different ways. \\
First, regularity in the direction of the potential, that is that we study $\be\mapsto \CP(\be)$. Conversely, non-regularity yields different order of phase transitions as seen above. \\
Another point is the G\^{a}teaux-differentiability  of the functional $\CP(\phi+.)$ on $\CC(\S)$. In \cite[Cor. 2]{walters-gateaux}, it is showed that G\^{a}teaux-differentiability  at $\phi$ is equivalent to uniqueness of the equilibrium state for $\phi$. Of course this later point is more subtile that the regularity for $\CP(\be)$, but it was however communally expected that analyticity for $\CP(\be)$ would insure uniqueness of the equilibrium state. We prove here that it is not true.

\medskip\noindent
{\bf Theorem B.}\\
{\it There exist irreducible subshifts of finite type and continuous potentials such that their pressure function are analytic on some interval $]\be_{c}',+\8[$ but there coexist several equilibrium states.

There exists an infinite-dimensional space of functions $\varphi$ such that for every $\be>\be_{c}'$, $\varphi\mapsto \CP(\phi+\varphi)$ is G\^{a}teaux-differentiable in the $\varphi$-direction. 
}

We remind that a function $f$ is G\^{a}teaux-differentiable at $x$ in the direction $y$ if 
$$\lim_{t\to 0}\frac{f(x+ty)-f(x)}{t}$$
exists.

%
%
%
%
%

\subsection{Overview of the paper}
The main tool is the notion of local equilibrium state as it was introduced in \cite{leplaideur1} and developed in further works of the author. We briefly recall in Appendix \ref{sec-locthermo} the principal steps. 

In Section \ref{sec-transiglobale} we prove main of the results concerning Theorem A and the phase transition for $\CP(\be)$. In Section \ref{sec-transicp234} we prove the ``small'' phase transition for $\S_{234}$ and finish the proof of Theorem A. 

In these both sections the main issue is to define and understand Figure \ref{fig-curves} (p. \pageref{fig-curves}). The special curves $\wt Z_{c}(\be)$, $\lambda_{Z,\be,[1]}$ and $\lambda_{Z,\be,[32]}$ appear on that figure. Most of the proofs of lemmas are easily under stable if one keep Figure \ref{fig-curves}  in mind. 

In Section \ref{sec-theoB} we prove Theorem B. We explain how to adapt the results from the proof of Theorem A to the new situation. 


\section{The phase transition for $\CP(\be)$}\label{sec-transiglobale}

\subsection{The induced operator on $[1]$}
The first return time map has been widely used in dynamical systems to understand the ergodic properties of a system. In this direction, we introduced in \cite{leplaideur1}a generalisation of the classical transfer operator which allow us to obtain information in cases that the usual theory breaks down. Denote by  $\tau_{[1]}(x)$  the first return time into the cylinder $[1]$. Given $Z \in  \R$ the \emph{Induced Transfer Operator on the cylinder $[1]$} for the first return time map and for the potential
$$\beta\phi(x)+\ldots +\beta\phi\circ \s^{\tau_{[1]}(x)}(x)-\tau_{[1]}(x)Z,$$
is defined by
$$\CL_{Z,\be,[1]}(\psi)(x):=\sum_{y: \, g(y)=x}e^{S_{\tau(y)}(\be.\phi)(y)-Z \tau_{[1]}(y)}\psi(y).$$
A counting argument yields
\begin{eqnarray*}
\CL_{Z,\be,[1]}(\BBone_{[1]})(x)&=&\sum_{n=1}^{+\8}e^{-n\beta\alpha-nZ+(n-1)\log L}+
\sum_{n=1}^{+\8}\left(\frac1{n+1}\right)^{\beta}e^{-nZ}\cdot e^{-\alpha\beta-Z}\times\\
&&
\left[\sum_{k=0}^{+\8}\left(
\sum_{n=1}^{+\8}\left(\frac1{n+1}\right)^{\beta}e^{-nZ}\,\cdot\sum_{n=1}^{+\8}\left(\frac1{n+1}\right)^{\eps\beta}L_{n}(\beta,[3])e^{-nZ}\right)^{k}\right],
\end{eqnarray*}
where
\begin{equation}
\label{equ-ope3}
L_{n}(\beta,[3])=\sum_{w\in\{3,4\}^{n},w_{0}=w_{n-1}=3}e^{\beta\varphi(w)}.
\end{equation}
Indeed, the first summand corresponds to the points  $x=(1x_2x_3 \dots x_{n-1} \dots)$ such that for every $ i \in \{2,3, \dots, n-1 \}$ we have $x_i \in \{1_{1},\ldots, 1_{L}\}$.  For those points we have that $\beta\phi(x)+\ldots +\beta\phi\circ \s^{\tau^{1}(x)}(x)-\tau_{[1]}(x)Z= -n\beta \alpha -nZ$. Moreover, since the system restricted to $\{1,1_{1},\ldots, 1_{L}\}$ is a full-shift on $L$ symbols there are exactly $L^{n-1}$ different words of the form just described. Hence the factor $(n-1)\log L$.

The second summand considers the points of the form $x=(1 2, 2,2,  \dots 2 1 \dots)$. In this case $e^{-\beta \alpha -Z }$ takes into account the potential at $x$ and the contribution for the strain go 2's yields the term $\disp\left(\frac1{n+1}\right)^{\be}$.

Finally, we have to consider the points in $[1]$ that before returning to the cylinder $[1]$ visit the cylinders $[3],[4]$.  They must first visit $[2$. These words are of the form
$$1(\text{string of 2's})\left(\text{intermittence of strings of 3's or 4's and strings of 2's}\right)1\ldots$$

Note that,

\begin{eqnarray*}
L_{n}(\be,[3])&=& e^{n\be\gamma}\sum_{k=0}^{n-2}e^{k\be\delta}{n-2\choose k}
 \times \text{possible $4$'s in the word $w$}\\
&=& e^{n\be\ga}(1+e^{\be\delta})^{n-2}=\frac{\left(e^{\ga\be}+e^{(\ga+\delta)\be}\right)^{n}}{(1+e^{\delta\be})^{2}}= \frac1{(1+e^{\de\be})^{2}}e^{n\CP_{34}(\be)}.
\end{eqnarray*}

%
%

Following Section \ref{sec-locthermo} we need to determine the curve $Z_{c}(\be)$ associated to $\CL_{{Z,\be,[1]}}$. 
For simplicity we set 
\begin{eqnarray*}
\S_{1}=\S_{1}(Z,\be)&:=&\sum_{n=1}^{+\8}e^{-n\beta\alpha-nZ+(n-1)\log L}\\
\S_{2}=\S_{2}(Z,\be)&:=& \sum_{n=1}^{+\8}\left(\frac1{n+1}\right)^{\beta}e^{-nZ}\\
\S_{3}=\S_{3}(Z,\be)&:=&\frac1{(1+e^{\de\be})^{2}}\sum_{n=1}^{+\8}\left(\frac1{n+1}\right)^{\eps\beta}e^{n(\CP_{34}(\be)-Z)}.
\end{eqnarray*}

With these notations we get  for every $x$ in $[1]$
\begin{eqnarray*}
\CL_{{Z,\be,[1]}}(\BBone_{[1]})(x)&=& \S_{1}+\S_{2}e^{-\al\be-Z}\sum_{k=0}^{+\8}(\S_{2}\S_{3})^{k}\\
&=& \S_{1}+\frac{\S_{2}e^{-\al\be-Z}}{1-\S_{2}\S_{3}},
\end{eqnarray*}
which makes sense only if $\S_{2}\S_{3}<1$.

\subsection{Critical value $Z_{c}(\be)$ for $\CL_{{Z,\be,[1]}}$}

The quantity $\CL_{{Z,\be,[1]}}(\BBone_{[1]})$ is well-defined if and only if 
\begin{subeqnarray}
 \S_{1}<+\8,\slabel{equ-condiS1}\\
 \S_{2}\S_{3}<1.\slabel{equ-condiS2S3}
\end{subeqnarray}

Condition \eqref{equ-condiS1} is satisfied only if $Z>\log L-\al\be$. Let us now study Condition \eqref{equ-condiS2S3}.  
First, we emphasize that a necessary condition is $Z\ge \CP_{34}(\be)$. Now we have:
\begin{lemma}
\label{lem-condiS2S3}
For every values of the parameters, there exists a unique $\be_{1}$ which is positive such that for every $0\le \be<\be_{1}$ there exists an unique $\wt Z_{c}(\be)$ satisfying
\begin{enumerate}
\item $\wt Z_{c}(\be)> \CP_{34}(\be)$,\\
\item $\S_{2}\S_{3}<1$ for every $Z>\wt Z_{c}(\be)$,\\
\item $\S_{2}\S_{3}>1$ for every $\CP_{34}(\be)\le Z<\wt Z_{c}(\be)$,\\
\item $\S_{2}\S_{3}=1$ for every $Z=\wt Z_{c}(\be)$.
\end{enumerate}
\end{lemma}
\begin{proof}
Note that $\S_{2}$ and $\S_{3}$ decreases in $Z$ for fixed $\be$. For a fixed $\be$, the unicity of $\wt Z_{c}(\be)$  such that 
$$\S_{2}(Z,\be)\S_{3}(Z,\be)=1$$
is thus proved (if it exists !). 
Note that 
both $\S_{2}$ and $\S_{3}$ go to 0 if $Z$ goes to $+\8$. Then, existence of $\wt Z_{c}(\be)$ follows from  the value of $\S_{2}(0,\be)\S_{3}(0,\be)$. If it is larger than 1 then $\wt Z_{c}(\be)$ exits, if it is smaller than 1, then $\wt Z_{c}(\be)$ does not exist.
Now,
$$\S_{2}(\CP_{34}(\be),\be)\S_{3}(\CP_{34}(\be),\be)=\sum_{n=1}^{+\8}\left(\frac1{n+1}\right)^{\be}e^{-n\CP_{34}(\be)}\frac{\zeta(\eps\be)-1}{(1+e^{\de\be})^{2}}.$$
 The function $\be\mapsto\CP_{34}(\be)$ increases in $\be$, thus $\be\mapsto\S_{2}(\CP_{34}(\be),\be)\S_{3}(\CP_{34}(\be),\be)$ decreases in  $\be$. It goes to $+\8$ if $\be$ goes to 0 and to $0$ if $\be$ goes to $+\8$.

 Therefore, there exists a unique $\be_{1}$ such that 
 \begin{equation}
\label{equ-beta1}
\S_{2}(\CP_{34}(\be_1),\be_1)\S_{3}(\CP_{34}(\be_1),\be_1)=\sum_{n=1}^{+\8}\left(\frac1{n+1}\right)^{\be_1}e^{-n\CP_{34}(\be_1)}\frac{\zeta(\eps\be_1)-1}{(1+e^{\de\be_1})^{2}}=1
\end{equation}
\end{proof}
\begin{remark}
\label{rem-epsbe}
We point out that $\eps\be_{1}>1$ because the Zeta function does converge. 
$\blacksquare$\end{remark}

Note that for every $\be<\be_{1}$, $\wt Z_{c}(\be)>\CP_{34}(\be)$ and is given by the following implicit formula
$$\sum_{n=1}^{+\8}\left(\frac1{n+1}\right)^{\be}e^{-n\wt Z_{c}(\be)}\frac{1}{(1+e^{\de\be})^{2}}\sum_{n=1}^{+\8}\left(\frac1{n+1}\right)^{\eps\be}e^{-n(\wt Z_{c}(\be)-\CP_{34}(\be))}=1.$$ 
It shows $\wt Z_{c}(\be)$ is analytic in $\be$ for $\be<\be_{1}$. 

Another consequence of Lemma \ref{lem-condiS2S3} is that for $\be>\be_{1}$, Condition \eqref{equ-condiS2S3} holds for every $Z\ge \CP_{34}(\be)$, the case $\be=\be_{1}$ being the critical one\footnote{Actually, for $\be=\be_{1}$, \eqref{equ-condiS2S3} holds for every $Z> \CP_{34}(\be_{1})$. }.

\medskip
Consequently, Conditions \eqref{equ-condiS1} and \eqref{equ-condiS2S3} hold if and only if
\begin{enumerate}
\item for $\be\le \be_{1}$, $Z>\log L-\al\be$ and $Z>\wt Z(\be)$,
\item for every $\be>\be_{1}$, $Z>\log L-\al\be$ and $Z\ge \CP_{34}(\be)$.
\end{enumerate}

\paragraph{\bf Hypotheses.} 
In order to simplify the determination of $Z_{c}(\be)$, we make the assumption $L=1$. This assumption does not affect the results neither the proofs. We shall see below that $\wt Z_{c}$ is increasing so as $\CP_{34}$ and $\be\mapsto \log L-\al\be$ is decreasing. Therefore, there eventually exists some $\be'$ such that for every $\be>\be'$, $Z_{c}(\be)=\CP_{34}(\be)$ (see Figure \ref{fig-curves}). 
Actually, the assumption allows us to consider $\be'=0$, and then just to consider relative positions with respect to $\CP_{234}(\be)$ (which is equal $\wt Z_{c}(\be)$ for $\be<\be_{1}$) and $\CP_{34}(\be)$. 

\medskip
With this assumption, $Z_{c}(\be)=\wt Z_{c}(\be)$ if $\be\le\be_{1}$ and $\CL_{{Z,\be,[1]}}(\BBone_{[1]})$ diverges  for $Z=Z_{c}(\be)$; $Z_{c}(\be)=\CP_{34}(\be)$ if $\be>\be_{1}$ and $\CL_{{Z,\be,[1]}}(\BBone_{[1]})$ converges for $Z=Z_{c}(\be)$.

\subsection{Spectral radius of $\CL_{{Z,\be,[1]}}$}
Due to the form of the potential, the spectral radius of $\CL_{{Z,\be,[1]}}$ is given by 
\begin{equation}
\label{equ-lambdatoto1}
\lambda_{{Z,\be,[1]}}=\CL_{{Z,\be,[1]}}(\BBone_{[1]})(x)=\S_{1}+\frac{\S_{2}e^{-\al\be-Z}}{1-\S_{2}\S_{3}},
\end{equation}
for any $x$ in $[1]$. 

We are interested by the level curve $\lambda_{{Z,\be,[1]}}=1$ because it partially gives the implicit function $Z=\CP(\be)$. 

\begin{lemma}
\label{lem-lambdatoto1=1}
There exists $\be_{c}>\be_{1}$ such that for every $\be<\be_{c}$, there exists a unique $Z>Z_{c}(\be)$ such that $\lambda_{{Z,\be,[1]}}=1$. 

Moreover, for every $\be>\be_{c}$ and for every $Z\ge Z_{c}(\be)=\CP_{34}(\be)$, $\lambda_{{Z,\be,[1]}}<1$. 
\end{lemma}
\begin{proof}
We emphasize that $\S_{1}$, $\S_{2}$ and $\S_{3}$ decreases in $Z$. 
Let us first consider the case $\be\le \be_{1}$. For a fixed $\be$, if $Z\downarrow \wt Z_{c}(\be)$, $\lambda_{{Z,\be,[1]}}$ increases to $+\8$ because $\S_{2}\S_{3}$ goes to 1. On the other hand, if $Z$ goes to $+\8$, $\lambda_{{Z,\be,[1]}}$ goes to 0. 

Therefore, there exists a unique $Z$ such that $\lambda_{{Z,\be,[1]}}=1$. 
 
Let us now assume $\be>\be_{1}$. Again, $\lambda_{{Z,\be,[1]}}$ goes to 0 if $\be$ goes to $+\8$. Existence of a solution for 
$$\lambda_{{Z,\be,[1]}}=1,$$
is thus a consequence of inequality 
\begin{equation}
\label{equ-lambdatoto1P34}
F(\be):=\sum_{n=1}^{+\8}e^{-n\al\be-n\CP_{34}(\be)}+\frac{\sum_{n=1}^{+\8}\left(\frac1{n+1}\right)^{\be}e^{-n\CP_{34}(\be)}e^{-\al\be-\CP_{34}(\be)}}{1-\frac1{(1+e^{\de\be})^{2}}\sum_{n=1}^{+\8}\left(\frac1{n+1}\right)^{\be}e^{-n\CP_{34}(\be)}\sum_{n=1}^{+\8}\left(\frac1{n+1}\right)^{\eps\be}}\ge 1.
\end{equation}
As $\be\mapsto \CP_{34}(\be)$ increases, we claim that $F(\be)$ decreases in $\be$. If $\be\downarrow \be_{1}$, $\S_{2}\S_{3}$ goes to 1 and then $F(\be)$ goes to $+\8$. If $\be\to+\8$, $F(\be)\to0$. 

Consequently, there exists an unique $\be_{c}$, such that $F(\be_{c})=1$. We have $\be_{c}>\be_{1}$ because $\lim_{\be\downarrow \be_{1}}F(\be)=+\8$. 

For $\be<\be_{c}$, the implicit equation $\lambda_{{Z,\be,[1]}}=1$ has an unique solution, $Z=\CP(\be)$. For $\be>\be_{c}$ it has no solution. 
 \end{proof}

\subsection{Thermodynamic formalism for $\be<\be_{c}$}
If $\be<\be_{c}$ we simply use Fact \ref{fact-spectrad=1equil} in Appendix \ref{sec-locthermo}. As $\CP(\be)$ is given by an implicit function inside the interior of the domain of analyticity in both variables, $\be\mapsto\CP(\be)$  is analytic for $\be<\be_{c}$. 

By construction, $\CP(\be)>Z_{c}(\be)\ge \CP_{34}(\be)$. 

\medskip
We can now study the influence of the assumption $L=1$. 
If $L$ is not supposed to be equal to 1 then, $Z_{c}(\be)$ is, by definition larger than $\log L-\al\be$. For $Z\downarrow Z_{c}(\be)$, if $\log L-\al\be>\wt Z_{c}(\be)$ or/and $\log L-\al\be>\CP_{34}(\be)$, then $\lambda_{{Z,\be,[1]}}$ goes to $+\8$ and the same reasoning than above works. 

On  the other hand, $\log L-\al\be$ decreases in $\be$ to $-\8$, whereas $\CP_{34}(\be)$ increases in $\be$ to $+\8$. Therefore, there eventually exists some $\be'$ such that for every $\be>\be'$ $Z_{c}(\be)=\wt Z_{c}(\be)$ (if $\be<\be_{1}$) or $Z_{c}(\be)=\CP_{34}(\be)$ (if $\be>\be_{1}$). Then, the reasoning is the same. 

Hence, the influence of $L$ only shift the phase transition to higher $\be_{c}$. As a by-product, this shows that $\eps\be_{c}$ can be made as big as wanted if $L$ increases.

\begin{lemma}
\label{lem-pgoeszcbetac}
$\disp\lim_{\be\to\be_{c}^{-}}\CP(\be)=\CP_{34}(\be_{c})$.
\end{lemma}
\begin{proof}
Rememeber that  $\CP(\be)$ is given by the implicit formula $\lambda_{{Z,\be,[1]}}=1$. On the other hand, for every $\be_{1}<\be<\be_{c}$ and $Z=\CP_{34}(\be)$ we set
$$\lambda_{{Z,\be,[1]}}=F(\be).$$
Furthermore, $F(\be)>1$  for $\be<\be_{c}$ and goes to 1 if $\be\to \be_{c}$ (by definition of $\be_{c}$). 

As for any fixed $\be$, $Z\mapsto\lambda_{{Z,\be,[1]}}$ decreases, for $\be=\be_{c}$, $Z=\CP_{34}(\be)$ is the unique solution for 
$$\lambda_{{Z,\be,[1]}}=1, $$
thus $\CP(\be_{c})= \CP_{34}(\be_{c})$. 
\end{proof}

\subsection{Number of equilibrium states at $\be_{c}$ and thermodynamic formalism for $\be>\be_{c}$}

Note that due to Facts \ref{fact-specrad<1} and \ref{fact-casezc} in Appendix, for $\be>\be_{c}$ no equilibrium state can gives positive weight to $[1]$. At the transition, $\be=\be_{c}$, $\CP(\be_{c})=\CP_{34}(\be_{c})$ which yields that the unique equilibrium state in $\S_{34}$ for $\be.\phi$ is one equilibrium state for the global system. 
Then, existence of one equilibrium state giving positive weight to $[1]$ is related to the condition $\disp\left|\frac{\partial\CL_{{Z,\be,[1]}}(\BBone_{[1]})(x)}{\partial Z}\right|<+\8$.

\begin{lemma}
\label{lem-equizobi1}
If $\mu$ is an equilibrium state for $\be.\phi$ and $\mu([1])=0$, then $\mu([1_{1}])=\ldots=\mu([1_{L}])=0$. 
\end{lemma}
\begin{proof}
Assume $\mu$ is an equilibrium state for $\be.\phi$ and $\mu([1])=0$. As we always can consider an ergodic component of $\mu$  it is equivalent to assume that $\mu$ is ergodic. 

Then, if $\mu([1_{1}]\cup\ldots\cup[1_{L}])>0$, by ergodicity, $\mu([1_{1}]\cup\ldots\cup[1_{L}])=1$. As the potential $\phi$ is constant on $[1_{1}]\cup\ldots\cup[1_{L}]$, $\mu$ is the measure with maximal entropy supported in $\{1_{1},\ldots,1_{L}\}^{\N}$ and $\CP(\be)=\log L-\be\al$. 

In that case, the measure of maximal entropy supported in $\{1,1_{1},\ldots,1_{L}\}^{\N}$  has a pressure $\log (L+1)-\be\al>\CP(\be)$ which is impossible. 
\end{proof}

Consequently, for $\be>\be_{c}$, any the equilibrium state has support in $\{2,3,4\}^{\N}\cap \S$

\begin{proposition}
\label{prop-equil-epsbeta}
If $\eps\be_{c}>2$, then there are at least two equilibrium states for $\be=\be_{c}$. If $\eps\be_{c}\le2$, then no equilibrium state for $\be=\be_{c}$ gives positive weight to $[1]$. 
\end{proposition}
\begin{proof}
We remind than $$\lambda_{{Z,\be,[1]}}=\CL_{{Z,\be,[1]}}(\BBone_{[1]})(x),$$
for every $x$ in $[1]$ and Equality \eqref{equ-lambdatoto1} is 
$$\lambda_{{Z,\be,[1]}}=\S_{1}(Z,\be)+\frac{\S_{2}(Z,\be)e^{-\al\be-Z}}{1-\S_{2}(Z,\be)\S_{3}(Z,\be)},$$
with $Z=\CP_{34}(\be)$. Therefore we have to compute $\disp\frac{\partial \lambda_{{Z,\be,[1]}}}{\partial Z}{}_{|Z=\CP_{34}(\be)}$. 
This quantity can be expressed in function of $\S_{1}$ $\S_{2}$ and $\S_{3}$ but also $\disp\frac{\partial \S_{1}}{\partial Z}$, $\disp\frac{\partial \S_{2}}{\partial Z}$ and $\disp\frac{\partial \S_{3}}{\partial Z}$. 

We point out that as  $\CP_{34}(\be)>0$ and we necessarily have $\log L-\al\be_{c}<\CP(\be_{c})$, the convergence in the series defining $\S_{1}$, $\S_{2}$ but also $\disp\frac{\partial \S_{1}}{\partial Z}$ and $\disp\frac{\partial \S_{2}}{\partial Z}$ are exponential. Therefore, the global convergence is equivalent to the convergence of $\disp\frac{\partial \S_{3}}{\partial Z}$, that is
\begin{equation}
\label{equ1-espbeta<2}
\sum_{n}\left(\frac1{n+1}\right)^{\eps\be}n<+\8.
\end{equation}
This holds if and only if $\eps\be_{c}>2$. 
\end{proof}

\begin{remark}
\label{rem-pluriequiltransiglobal}
If $\eps\be_{c}>2$, there exists a unique equilibrium state which gives positive weight to $[1]$. 
$\blacksquare$\end{remark}

\section{Phase transition for $\CP_{234}(\be)$. End of the proof of Theorem A}\label{sec-transicp234}

\subsection{First inequalities and obvious results}
One of the main difficulties is that we do not know, at that stage if the pressure for $\be>\be_{c}$ is strictly bigger than $\CP_{34}(\be)$ or not.

Let $\CP_{234}(\be)$ be the pressure for  the sub-system $\{2,3,4\}^{\N}\cap\S$ of points in $\S$ with no symbols in $\{1,1_{1},\ldots,1_{L}\}$ and for the potential $\be.\phi$.
As it is a subsystem of the global one, $\CP_{234}(\be)\le \CP(\be)$. Conversely, $\{3,4\}^{\N}$ is a subsystem of  $\{2,3,4\}^{\N}\cap\S$ and then $\CP_{234}(\be)\ge \CP_{34}(\be)$. Therefore 
$$\CP_{234}(\be_{c})=\CP(\be_{c})=\CP_{34}(\be_{c}).$$

The main question is to know if for $\be> \be_{c}$ one equilibrium state gives weight to the cylinder $[2]$ or not.

We recall that $\phi$ is continuous and the entropy is upper semi-continuous. Thus, there exists at least one equilibrium state, say $\wh\mu_{\be}$, in $\S_{234}$. 

\begin{lemma}
\label{lem-3measpos}
For every $\be$,  for every equilibrium state $\wh\mu_{\be}$, $\wh\mu_{\be}([3])>0$.
\end{lemma}
\begin{proof}
If $\wh\mu_{\be}([3])=0$, then any ergodic component of $\wh\mu_{\be}$ is either $\delta_{2^{\8}}$ or $\de_{4^{\8}}$. In
the first case the pressure is 0, in the second case it is $\be.(\de+\ga)$. In both cases the value is strictly lower
than $\CP_{34}(\be)\le \CP_{234}(\be)$. 
\end{proof}

\begin{lemma}
\label{lem-induc32bon}
Let $\wh\mu_{\be}$ be an equilibrium state for $\be.\phi$. Then $\wh\mu_{\be}([32])=0$ if and only if $\CP_{234}(\be)=\CP_{34}(\be)$. 
\end{lemma}
\begin{proof}
If $\wh\mu_{\be}([32])=0$, then $\s$-invariance immediately yields that any cylinder  of the form $[i_{0}i_{1}\ldots i_{n-1}32]$ has null $\wh\mu_{\be}$-measure. As $\wh\mu_{\be}([3])>0$, this shows that $\wh\mu_{\be}(\S_{34})=1$, thus $\CP_{234}(\be)\le \CP_{34}(\be)$. The converse inequality is true as recalled above. 
\end{proof}

For our purpose we will thus induce on the cylinder $[32]$. To avoid heavy notations, the first return will simply be denoted by $\tau$ and the first return map by
$T$.

\subsection{Induced operator in $[32]$}
Consider a point in $[32]$ say $x:=32x_{2}x_{3}\ldots$. Any $y$ satisfying $T(y)=x$ is of the form
$$y=3\underbrace{2\ldots 2}_{\text{at least one }2}3\omega32x_{2}x_{3}\ldots,$$
where $\omega$ is a word in 3 and 4.

If $x':=32x'_{2}x'_{3}\ldots$ $y:=32^{n}3\omega32x_{2}x_{3}\ldots$ and $y':=32^{n}3\omega32x'_{2}x'_{3}\ldots$, we emphasize that 
\begin{equation}\label{equ-boweninduc32}
S_{n+1+|\omega|+1}(\phi)(y')=S_{n+1+|\omega|+1}(\phi)(y)
\end{equation}
holds. This is the main interest to induce on $[32]$. 

Then, the family of  transfer operators for $T$ and $\be.\phi$ is defined by 
\begin{equation}
\label{equ-CL32}
\CL_{{Z,\be,[32]}}(\psi)(x):=\frac1{2^{\eps\be}}e^{\ga\be-Z}\sum_{n=1}^{+\8}\sum_{\omega}\left(\frac1{n+1}\right)^{\be}e^{-nZ}\left(\frac2{|\omega|+3}\right)^{\eps\be}e^{|\omega|\ga\be+(\#4\in\omega).\de\be-|\omega|Z}\psi(32^{n}\omega x),
\end{equation}
where $\psi$ belongs to $\CC^{0}([32])$, $x$ starts with 32\ldots, $\omega$ is a  (possibly empty) word with digits 3 and 4 and starting with 3 and $\#4\in\omega$ is the number of 4's in $\omega$. 

Due to \eqref{equ-boweninduc32}, the potential (for induction) satisfies condition (C2) in Appendix  \ref{sec-locthermo}, and then we have
by Lemma \ref{lem-specradiCL32}

$$\lambda_{{Z,\be,[32]}}=\CL_{{Z,\be,[32]}}(\BBone_{[32]})=\frac{1}{(1+e^{\de\be})^{2}}\sum_{n=1}^{+\8}\left(\frac1{n+1}\right)^{\be}e^{-nZ}\sum_{m=1}^{+\8}\left(\frac1{m+1}\right)^{\eps\be}e^{m(\CP_{34}(\be)-Z)}.$$

Now, the implicit equation 
$$\lambda_{{Z,\be,[32]}}=1,$$
is exactly realized for $Z=\wt Z_{c}(\be)$ (see Lemma \ref{lem-condiS2S3}) and holds if and only if $\be\le\be_{1}$. By definition of $\be_{1}$, for every $\be>\be_{1}$, 
$$\S_{2}\S_{3}<1,$$
and then Facts \ref{fact-spectrad=1equil}, \ref{fact-specrad<1} and \ref{fact-casezc} show that

$\bullet$ $\CP_{234}(\be)=\wt Z_{c}(\be)$ for $\be<\be_{1}$,

$\bullet$  there is a unique equilibrium state for $\be<\be_{1}$, and it is fully supported in $\S_{234}$,

$\bullet$ the unique equilibrium state for $\be>\be_{1}$ is the one in $\S_{34}$, 

$\bullet$ there are two equilibrium state for $\be=\be_{1}$ if and only if $\eps\be_{1}>2$. 

At that point, all the results stated in Theorem A are proved except  $\eps.\be_{1}<2$ and the fact that condition $\eps.\be_{c}\lesseqgtr2$ and  can be realized.

\begin{figure}[htbp]
\begin{center}
\includegraphics[scale=0.6]{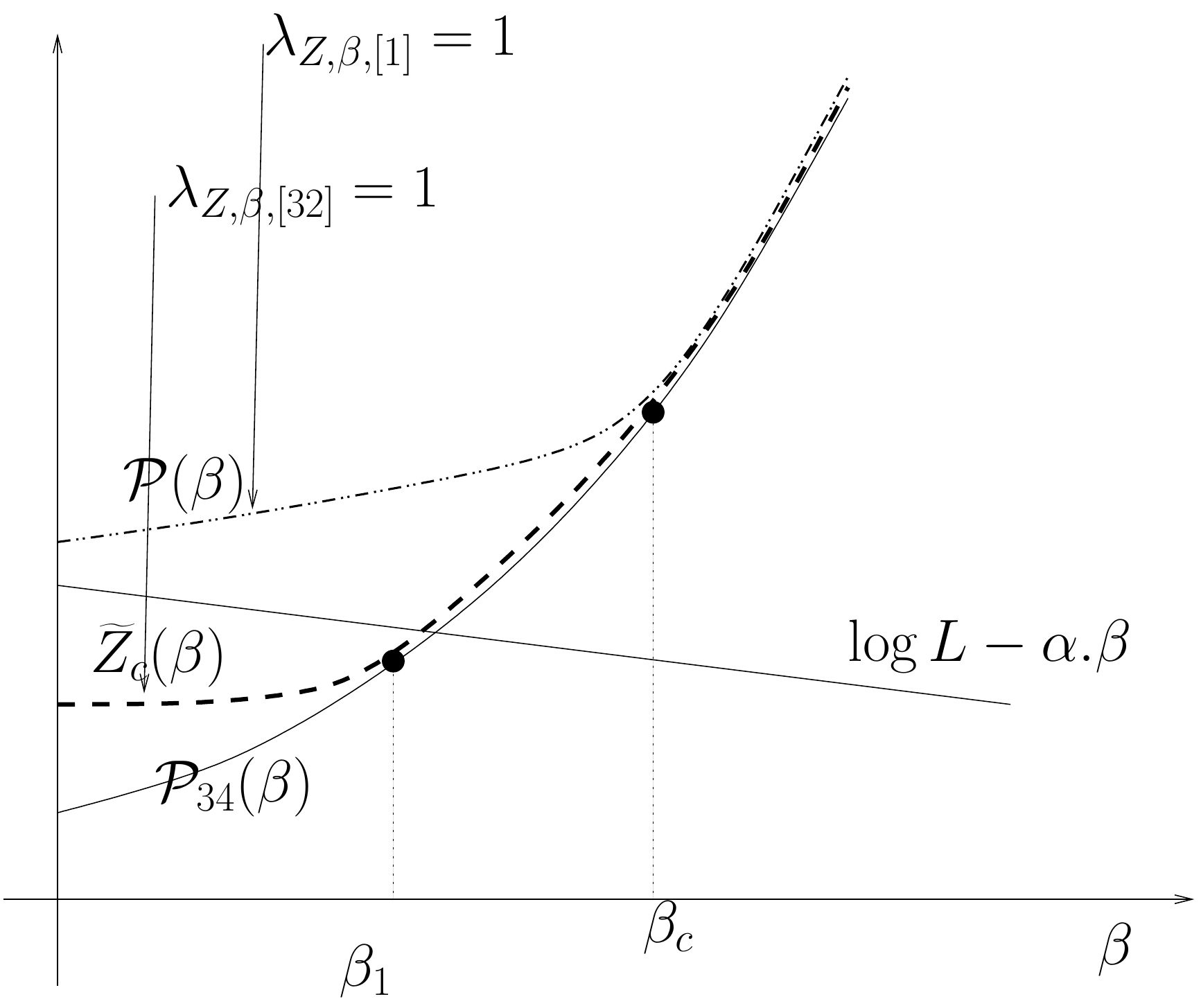}
\caption{Principal curves.}
\label{fig-curves}
\end{center}
\end{figure}

\subsection{End of the proof of Theorem A}
\subsubsection{Proof that $\eps\be_{1}<2$}

We remind that $\be_{1}$ is defined by the implicit formula \eqref{equ-beta1}:

$$\sum_{n=1}^{+\8}\left(\frac1{n+1}\right)^{\be_1}e^{-n\CP_{34}(\be_1)}\frac{\zeta(\eps\be_1)-1}{(1+e^{\de\be_1})^{2}}=1,$$
with $\CP_{34}(\be)=\ga\be+\log(1+e^{\de\be})$. 
Note that $\CP_{34}(\be)$ is always larger than $\log2$, thus for every choices of the parameters
$$\zeta(\eps\be_{1})\ge (1+e^{\de\be_{1}})^{2}+1>5.$$
Now, $\zeta(2)=\disp\frac{\pi^{2}}6$, which shows  $\eps\be_{1}<2$. 

\subsubsection{Values for $\eps\be_{c}$}
We remind that $\be_{c}$ is given by the implicit formula derived from  Inequality \eqref{equ-lambdatoto1P34}:

\begin{equation}
\label{equ-implicibetac}
\sum_{n=1}^{+\8}e^{-n\al\be_{c}-n\CP_{34}(\be_{c})+(n-1)\log L}+\frac{\sum_{n=1}^{+\8}\left(\frac1{n+1}\right)^{\be_{c}}e^{-n\CP_{34}(\be_{c})}e^{-\al\be_{c}-\CP_{34}(\be_{c})}}{1-\frac1{(1+e^{\de\be_{c}})^{2}}\sum_{n=1}^{+\8}\left(\frac1{n+1}\right)^{\be_{c}}e^{-n\CP_{34}(\be_{c})}\sum_{n=1}^{+\8}\left(\frac1{n+1}\right)^{\eps\be_{c}}}=1.
\end{equation}
We have already seen that increasing $L$ is a simple way to force $\eps\be_{c}>2$ to hold.

Remind that $\be_{c}>\be_{1}$ and $\eps.\be_{1}>1$. Then, assume that $\de\to+\8$, $\eps$ being fixed, this yields $\de\be_{c}\to+\8$. Reporting this in \eqref{equ-implicibetac}, the first summand  and the numerator  of the fraction go to $0$ if $\de\to+\8$. Consequently, the denominator must also tend to 0,  and as $\de\be_{c}$ tends to $+\8$, we must have 
$$\eps\be_{c}\to 1.$$
It is thus lower than $2$ if $\delta$ is sufficiently large.

\section{Proof of Theorem B}\label{sec-theoB}

For proving Theorem B we consider the next subshift of finite type:

\begin{figure}[htbp]
\begin{center}
\includegraphics[scale=0.5]{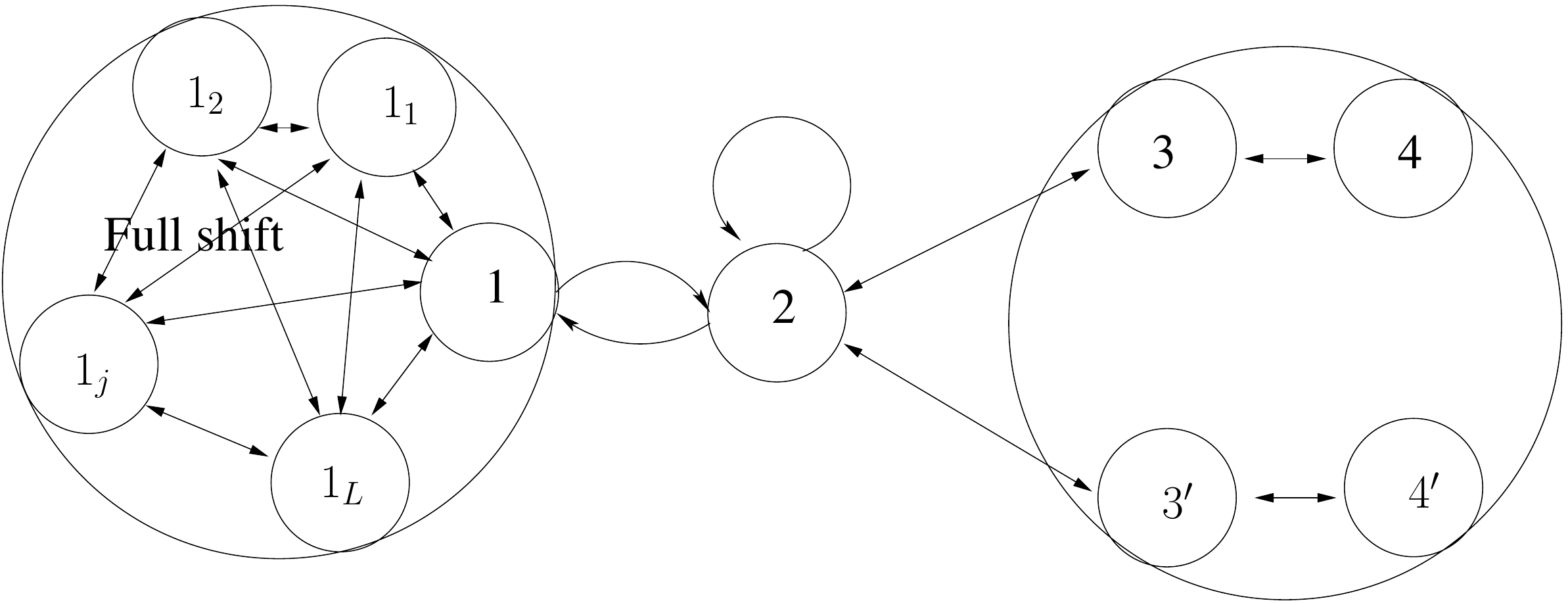}
\caption{Dynamics  for Theorem B}
\label{Fig-butterflyThB}
\end{center}
\end{figure}

\subsection{Phase transition for $\CP(\be)$}
We now explain how to adapt the results from the proof of Theorem A to  this new case. 
Inducing in $[1]$, orbits going to $[2]$ will/can visit $\S_{34}$ or $\S_{3'4'}$ before visiting $[2]$ again. More precisely, after a string of 2' we can either get a word of the form $3\omega3$ which $\omega$ a word with only 3 or 4 as digits, or a word of the form $3'\omega'3'$ with $\omega'$ a word with only 3' and 4'. 

By symmetry of the potential, any sum $\S_{3}$ as thus to be replaced by $2\S_{3}=\S_{3}+\S_{3}$, one for a string in 3 and 4 and one for a string in 3' and 4'. 

Consequently the condition \eqref{equ-condiS2S3} has to be replaced by 
\begin{equation}
\label{equcondiS2S3double}
\S_{2}\S_{3}<\frac12.
\end{equation}
The implicit formula (similar to \eqref{equ-beta1}) we have to consider is 
$$\S_{2}(\CP_{34}(\be_{2}),\be_{2})\S_{3}(\CP_{34}(\be_{2}),\be_{2})=\frac12,$$
where $\be_{2}$ replaces $\be_{1}$:
for $\be<\be_{2}$, $\wt Z_{c}(\be)$ is strictly larger than $\CP_{34}(\be)$. 
For $\be\ge \be_{2}$,  $\wt Z_{c}(\be)= \CP_{34}(\be)$. 

Similarly, the new value for the spectral radius $\lambda_{{Z,\be,[1]}}$ is 
$$\S_{1}+\frac{\S_{2}e^{-\al\be-Z}}{1-2\S_{2}\S_{3}},$$
and there exists $\be'_{c}>\be_{2}$ such that for every $\be>\be'_{c}$, $\lambda_{{Z,\be,[1]}}<1$. Then, for $\be>\be_{c}'$, $\CP(\be)=\CP_{2343'4'}(\be)$, and any equilibrium state has support 
in $\S_{2343'4'}$ which is $\S$ restricted to words without digits $1,1_{1},\ldots 1_{L}$. 

\subsection{Phase transition for $\CP_{2343'4'}(\be)$}
Again, we shall induce on the cylinder $[32]$. In that new case, the orbit leaving $[32]$ and then returning back to $[32]$ have the form:
$$32(\text{ string of 2's})\left(\text{ intermittence of strings of 3' 's or 4' 's and strings of 2's}\right)(\text{ string of 3's or 4's})32.$$
This yields that the new equation to consider for $\lambda_{{Z,\be,[32]}}$ is 
\begin{eqnarray*}
\lambda_{{Z,\be,[32]}}&=&
\S_{2}(\CP_{34}(\be),\be)\left(\sum_{n=0}^{+\8}(\S_{2}(\CP_{34}(\be),\be)\S_{3}(\CP_{34}(\be),\be))^{n}\right)\S_{3}(\CP_{34}(\be),\be)\\
&=&\frac{\S_{2}(\CP_{34}(\be),\be)\S_{3}(\CP_{34}(\be),\be)}{1-\S_{2}(\CP_{34}(\be),\be)\S_{3}(\CP_{34}(\be),\be)},
\end{eqnarray*}

where we took into account the symmetry for $\S_{34}$ and $\S_{3'4'}$. 

Note that $z\mapsto\disp \frac{z}{1-z}$ increases for $z<1$, and then 
$$\lambda_{{Z,\be,[32]}}<1\iff \S_{2}(\CP_{34}(\be),\be)\S_{3}(\CP_{34}(\be),\be)<\frac12.$$
This shows that $\be_{2}$ is a transition parameter for $\S_{2343'4'}$: for $\be<\be_{2}$ there exists a unique equilibrium state in $\S_{2343'4'}$ and it is fully supported. For $\be>\be_{2}$, no equilibrium state gives weight to $[32]$. For symmetric reason, no equilibrium state gives weight to $[3'2]$, and thus, there are two equilibrium states  which are the ones in $\S_{34}$ and in $\S_{3'4'}$.

For $\be>\be_{c}'$, there is no more a single global equilibrium state (for $\S$ and $\CP(\be)$) but these two ``smaller'' equilibria  in $\S_{34}$ and in $\S_{3'4'}$. Nevertheless, $\CP(\be)=\CP_{34}(\be)$ is analytic for $\be>\be'_{c}$.

\subsection{G\^{a}teaux-differentiability in other directions}

To use vocabulary from \cite{iommi-todd}, the potential $\phi$ is a kind of grid function: it is constant on cylinders of the form $[1]$, $[1_{i}]$, $[2^{n}*]$, $[3\omega32]$ with $\omega\in\{3,4\}^{n}$ for some $n$ and $[3'\omega3'2]$ with $\omega\in\{3',4'\}^{n}$ for some $n$.  

Let $\CV$ be the set of such functions, which are in addition H\"older continuous and totally symmetric in $3\leftrightarrow 3'$ and $4\leftrightarrow 4'$. $\CV$ is  infinite-dimension in $\CC(\S)$.

 It we pick some $\varphi$ in $\CV$, for $\be>\be_{c}'$, and for $t\in (-\eta,\eta)$ with $\eta\approx 0^{+}$, the spectrums for the  induced transfer operators for $\be.\phi+t.\varphi$ are the same than for $\be.\phi$.

As things are totally symmetric in $\S_{34}$ or $\S_{3'4'}$, there will still be two equilibrium states and the pressure is differentiable in direction $\varphi$ because we are into the domains with spectral radiuses $<1$. 

\begin{remark}
\label{rem-extention gateau}
It is actually highly probable that $P(\be.\phi+.)$ is G\^{a}beaux differentiable in the direction of any $\varphi$ which is H\"older and totally symmetric in $3\leftrightarrow 3'$ and $4\leftrightarrow 4'$ (and not necessarily a grid function). 
$\blacksquare$\end{remark}



\appendix
\section{some recall on induced transfer operator}\label{sec-locthermo}

We recall here some results from the construction of local equilibrium state as it was done in \cite{leplaideur1} and develop by the author in later works. 
We emphasize that the discussion $\lambda_{Z_{c},[i]}\gtreqqless1$ and $\disp\left|\frac{\partial\CL_{Z}}{\partial Z}\right|\le+\8$ is very similar to the cases of positive/null recurrence and transience in Sarig's work (see \eg \cite{sarig3,cyr-sarig}). 

\subsection{Local equilibrium states}
We consider a  one-sided subshit of finite type $\wh\S$ and a continuous potential $\varphi:\wh\S\to\R$. We consider a symbol $[i]$ and the associated cylinder $[i]$. Then, we consider the first return time  $\tau$ in $[i]$ and the first return map $g$. Even if the map $g$ is not defined everywhere, the main property is that inverse-branches are well-defined. 

The notion of local equilibrium state follows from the next question:  is there a way to do thermodynamical formalism for $([i],g)$ and some ``good'' potential such that the measure obtained in $[i]$,  invariant  for $g$ is the conditional measure of the/a equilibrium state for $\varphi$ ? 

For that purpose we consider the induced operator on $[i]$ depending on a parameter $Z$:
$$\CL_{Z}(\psi)(x):=\sum_{y\, g(y)=x}e^{S_{\tau(y)}(\varphi)(y)-Z.\tau(y)}\psi(y).$$
We recall that $\tau(y)$ is the return time: $g(y)=\s^{\tau(y)}(y)$. The parameter $Z$ is put to make the series converge. 

Namely, if the potential $\varphi$ satisfies the local Bowen condition:
$$\exists\, C\ \forall\, y,\, y'\in [ij_{1}\ldots j_{n-1}i],\ \left|S_{n}(\varphi)(y)-S_{n}(\varphi)(y')\right|<C,$$
where the symbols $j_{k}$ are different from $i$, then, there exists some critical $Z_{c}$ such that for every $\psi$ continuous and for every $x$, $\CL_{Z}(\psi)(x)$ converges for $Z>Z_{c}$ and diverges for $Z<Z_{c}$. 

Now, it is well-known that the Bowen condition yields good spectral properties for $\CL_{Z}$ (see  \eg \cite{baladibook}) and then existence and uniqueness of an equilibrium state for the system $([i],g)$ and the potential $S_{\tau(.)}(\varphi)(.)-Z.\tau(.)$ (for $Z>Z_{c}$). If it is denoted by $\mu_{Z,[i]}$, then the pressure is $\log\lambda_{Z,[i]}$, where $\lambda_{Z,[i]}$ is the spectral radius of $\CL_{Z}$. The measure $\mu_{Z,[i]}$ is referred to as a \emph{local equilibrium state}.

 We remind that the local Bowen condition follows from one of the 2 next conditions:
\begin{itemize}\label{condi-clz}
\item[(C1)---] the potential $\varphi$ is H\"older continuous,
\item[(C2)---] the potential $\varphi$  is such that $\CL_{Z}(\BBone_{[i]})(x)$ does not depend on $x$.
\end{itemize}

\subsection{Relation with global equilibrium state}

For every $Z>Z_{c}$, there exists a unique $\s$-invariant probability in $\wh\S$ $\wh\mu_{Z,[i]}$ such that its conditional measure (conditionally to $[i]$) is $\mu_{Z,[i]}$. Then, it turns out that $\wh\mu_{Z,[i]}$ is the unique equilibrium state for $\varphi-\lambda_{Z,[i]}\BBone_{[i]}$. 
Indeed, we have  for every $\s$-invariant  probability $\nu$
$$h_{\nu}(\s)+\int\varphi\,d\nu\le Z+\lambda_{Z,[i]}\nu([i]),$$
with equality if and only if $\nu=\wh\mu_{Z,[i]}$.

%

Thus, the unique equilibrium state for $\varphi$ (if it exists) is morally obtained by opening out the measure $\mu_{Z,[i]}$ for $Z=\CP(\varphi)$, or presumably equivalently, for the unique $Z$ such that $\lambda_{Z,[i]}=1$\footnote{This unique $Z$ turns out to be the pressure of $\varphi$.}

\medskip
However, one important condition to be checked is $Z_{c}<\CP(\varphi)$. 
If this condition holds, then the unique global equilibrium state is exactly corresponding to the unique local equilibrium state for $Z=\CP(\varphi)$. 

An convexity argument shows that $Z\mapsto Z-\lambda_{Z,[i]}\wh\mu_{Z,[i]}$ is maximal, either for $Z$ such that $\lambda_{Z,[i]}=1$, and in that case we get a global maximum, or for $Z=Z_{c}$ if for any $Z\ge Z_{c}$, $\lambda_{Z,[i]}\le1$. 

\begin{fact}\label{fact-spectrad=1equil}
 If there exists $Z>Z_{c}$ such that $\lambda_{Z,[i]}=1$, then necessarily $Z=\CP(\varphi)$, local \footnote{for $Z=\CP(\varphi)$} and global equilibrium are unique  and coincide up to conditioning to $[i]$. 
 \end{fact}
Note that in that case\footnote{The philosophy could be that everything is good at the same moment.}, $Z_{c}<\CP(\varphi)$.

\subsection{Detection or non-detection of an equilibrium state by the induction}
A critical case is when $Z_{c}=\CP(\varphi)$. In that case two situations may happen.

\begin{fact}
\label{fact-specrad<1}
If $\varphi$ satisfies some conditions yielding existence and uniqueness of the local equilibrium state, if for every $Z\ge Z_{c}$, $\lambda_{Z,[i]}<1$, then, no equilibrium state in the whole system $(\wh\S,\s)$ gives weight to the cylinder $[i]$. 
\end{fact}

We recall that under assumption on the kind (C1) or (C2), the local equilibrium state  is of the form $d\mu_{Z,[i]}=H_{Z,[i]}d\nu_{Z,[i]}$, where $\nu_{Z,[i]}$ is the eigen-probability measure for the dual operator $\CL^{*}_{Z,[i]}$ and $H_{Z,[i]}$ the eigen-function for $\CL_{Z}$, both associated to the spectral radius $\lambda_{Z}$. Uniqueness is obtained via the normalization 
$$\int H_{Z,[i]}\,d\nu_{Z,[i]}=1.$$
We point out that the function $H_{Z,[i]}$ is continuous and positive (otherwise the mixing property would show it is null). In other words, both measures $\mu_{Z,[i]}$ and $\nu_{Z,[i]}$ are equivalent. 

Now, such a measure is the restriction and renormalization of a global invariant measure if and only off it satisfies
$$\int \tau(x)\, d\mu_{Z,[i]}(x)<+\8.$$
Equivalence of the measures implies that this last condition is equivalent to $\disp \left|\frac{\partial \CL_{Z}(\BBone_{[i]})}{\partial Z}(x)\right|<+\8.$

\begin{fact}
\label{fact-casezc}
We emphasize that the construction of local equilibrium state holds for $Z=Z_{c}$ if $\CL_{Z_{c}}(\BBone_{[i]})(x)$ converges for every (or equivalently some) $x$. 

The fact that this local equilibrium state can be opened out to a global equilibrium state only depends if the expectation of its return time is finite or not. 
\end{fact}

Consequently, if for $Z=Z_{c}$, $\lambda_{Z_{c},[i]}=1$, the local equilibrium state gives a global equilibrium state for $\varphi$ if and only if $\disp \left|\frac{\partial \CL_{Z}(\BBone_{[i]})}{\partial Z}(x)\right|<+\8.$ There may be other equilibrium states, and they do not give weight to $[i]$. 

\subsection{The special case where Condition (C2) holds}\label{subsec-condiC2}

\begin{lemma}
\label{lem-specradiCL32}
If Condition (C2) holds, then for every $x$ in $[i]$,
$\disp\lambda_{Z,[i]}=\CL_{Z}(\BBone_{[i]})(x)$.
\end{lemma}
\begin{proof}
By definition of (C2)  $\CL_{Z}(\BBone_{[Z})(x)$ is a constant function in $x$.

Now, $\disp\log\lambda_{Z}=\limsup_\ninf\frac1n\log|||\CL_{Z}^n|||$, where
$$|||\CL_{Z}^n|||=\sup_{\psi\neq 0\in\CC^0}\frac{||\CL_{Z}^n(\psi)||_\8}{||\psi||_\8}.$$
Clearly, $\disp |||\CL_{Z}^n|||\ge ||\CL_{Z}^n(\BBone_{[Z]})||_\8$, but as it is a positive operator, for
every $\psi$,
$$||\CL_{Z}^n(\psi)||_\8\le ||\psi||_\8||\CL_{Z}^n(\BBone_{[Z]})||_\8.$$
Therefore, $\disp\log\lambda_{Z}=\limsup_\ninf\frac1n\log||\CL_{Z}^n(\BBone_{[Z]})||_\8$.

Now, $\CL_{Z}(\BBone_{[Z})$ is a constant function, thus for every $n$, 
$$\CL^{n}_{Z}(\BBone_{[Z]})=\left(\CL_{Z}(\BBone_{[Z]})(x)\right)^{n},$$ 
for any $x\in [i]$. 
\end{proof}

\bibliographystyle{plain}
\bibliography{bibliogodophase}

\newpage
Laboratoire de  Math\'ematiques de Bretagne Atlantique\\
UMR 6205\\
Universit\'e de Brest\\
6, avenue Victor Le Gorgeu\\
C.S. 93837, France \\
\texttt{Renaud.Leplaideur@univ-brest.fr}\\
\texttt{http://www.lmba-math.fr/perso/renaud.leplaideur}

\end{document}